\documentclass[11pt,a4paper,reqno]{amsart}
\usepackage{enumerate}
\usepackage[T1]{fontenc}
\usepackage{exscale,color,epsfig}
\usepackage{amsmath,amsthm,amssymb,amsfonts,latexsym,mathtools}
\usepackage[active]{srcltx}
\usepackage{latexsym}
\usepackage{pstricks,pst-grad}
\usepackage{tikz}
\usepackage{float}
\usepackage[textsize=tiny]{todonotes} 
\usepackage{xspace}

\usepackage[height=24.9cm,a4paper,hmargin={3cm,3cm},top=2.8cm]{geometry}

\newcommand{\Tr}{{\mathop{\mathrm{Tr} \,}}}
\DeclareMathOperator{\supp}{supp}
\DeclareMathOperator{\dist}{dist}
\newcommand{\tg}{\tilde{\gamma}}
\newcommand{\hg}{\hat{\gamma}}
\newcommand{\euler}{\mathrm{e}}

\newcommand{\hypU}{$(\mathbf{U})$\xspace}
\newcommand{\hypPi}{$(\mathbf{\Pi})$\xspace}
\newcommand{\hypNoPi}{$(\mathbf{No-\Pi})$\xspace}
\newcommand{\hypO}{$(\mathbf{O})$\xspace}

\newcommand{\ml}{m_{-}}
\newcommand{\mup}{m_{+}}

\newcommand{\vol}{\operatorname{vol}}

     \newcommand{\EE}{\mathbb{E}}
     \newcommand{\NN}{\mathbb{N}}
     \newcommand{\PP}{\mathbb{P}}
     \newcommand{\RR}{\mathbb{R}}
     \newcommand{\ZZ}{\mathbb{Z}}
     \newcommand{\cB}{\mathcal{B}}

\newtheorem{thm}{Theorem}

\newtheorem{cor}[thm]{Corollary}
\theoremstyle{definition}
\newtheorem{dfn}[thm]{Definition}
\theoremstyle{remark}
\newtheorem{rem}[thm]{Remark}

\begin{document}


\title[Wegner estimate with minimal support assumptions]{Wegner estimate and localisation for alloy type operators with minimal support assumptions on the single site potential}

\author[M.~T\"aufer]{Matthias T\"aufer}
\address{Lehrgebiet Analysis, Fakult\"at f\"ur Mathematik und Informatik, FernUniversit\"at in Hagen, Germany}

\author[I.~Veseli\'c]{Ivan Veseli\'c}
\address{Lehrstuhl IX, Fakult\"at f\"ur Mathematik,\,  TU Dortmund, Germany}
\urladdr{http://www.mathematik.tu-dortmund.de/lsix/}

\keywords{Spectral inequality, Uncertainty principle, Anderson model, Alloy Type model, Schr\"odinger operators, Wegner estimate, Localisation}

\begin{abstract}
We prove a Wegner estimate for alloy type models
merely assuming that the single site potential is lower bounded by a characteristic function of a thick set
(a particular class of sets of positive measure).
The proof exploits on one hand recently proven unique continuation principles or uncertainty relations for linear combinations of eigenfunctions of the Laplacian on cubes and on the other hand the well developed machinery for proving Wegner estimates.

We obtain a Wegner estimate with optimal volume dependence at all energies, and localisation near the minimum of the spectrum, even for some non-stationary random potentials.

We complement the result by showing that a lower bound on the potential by the characteristic function of
a thick set is necessary for a  (translation uniform) Wegner estimate to hold.
Hence, we have identified a sharp condition on the size for the support of random potentials that
is sufficient and necessary for the validity of Wegner estimates.

\end{abstract}

\maketitle


\maketitle

\section{Model and results}
We prove a Wegner estimate for continuum
random Schr\"o\-dinger operators with a very weak assumption on the supports of the single site potentials,
which turns out to be optimal.
The random potential needs not be stationary. Together with an initial scale estimate we conclude localisation for such models.

Until now, a fundamental assumption in this context has been that the potential either satisfies a \emph{covering condition}, see e.g.~\cite{MartinelliH-84,CombesH-94b,Kirsch-96}
or that it is uniformly positive on at least a non-empty \emph{open} set, see e.g.~\cite{CombesHK-03, CombesHK-07, RojasMolinaV-13, Klein-13, NakicTTV-18}.
We only ask the sum of potentials to be positive on a so-called \emph{thick set}:
\begin{dfn}
 \label{def:thick}
Let $\gamma\in (0,1]$, $a=(a_1,\ldots, a_d)\in(0,\infty)^d$, and
set $A_a := [0,a_1]\times \cdots \times[0,a_d]\subset\RR^d$.
A measurable set $S\subset\RR^d$ is called \emph{$(\gamma,a)$-thick} if
\begin{equation}
  \label{eq:definition_thick}
\operatorname{vol} \left( S\cap \left(x+A_a \right)  \right)
\geq
\gamma  \operatorname{vol} \left( A_a \right)
\quad
\text{for all}
\quad
x \in \RR^d
\end{equation}
and simply \emph{thick} if there exist $\gamma\in (0,1]$ and $a\in(0,\infty)^d$ such that~\eqref{eq:definition_thick} holds.
\end{dfn}

Thick sets are a generalization of periodic positive measure sets.
Thickness is the minimal condition on the size of a characteristic function necessary for a uncertainty relation of spectral
projectors to hold, cf.~\cite{EgidiV-20}, reformulating a criterion of \cite{Kovrijkine-00}.

We will study alloy type Hamiltonians on $L^2(\RR^d)$ satisfying the following assumption on the random variables:\\

\begin{enumerate}
 \item[\hypU]
 Let $[\ml, \mup] \subset \RR$ be an interval
 and $\Omega = \bigtimes_{j\in\ZZ^d}  [\ml, \mup]$ a probability space equipped with a product measure $\PP =\bigotimes_{j\in\ZZ^d} \mu_j$.
 Denote by $\pi_j\colon \Omega\to  [\ml, \mup]$ the coordinate projections such that $\{\pi_j\}_{j\in\ZZ^d},$ are a family of independent and uniformly bounded random variables.
 We assume that all $\pi_j$ are non-trivial.
\end{enumerate}
While most of our results would hold also under more general conditions on the random variables $\pi_j$, the spelled out condition is standard in the literature.
In fact, we want to focus our attention on a different building block of the random potential and identify the weakest possible condition on the support of the single site potentials, spelled out
in the following hypothesis:\\

\begin{enumerate}
 \item[\hypPi]
 There are $C_u,R>0$ and $p\in[2,\infty)$ with $p>d/2$ if $d>3$, such that
 for each $j \in \ZZ^d$, there is a measurable function $ u_j\colon \RR^d\to [0,\infty)$, supported in $B_R(j)$, the ball of radius $R$ around $j$, and satisfying
 $\Vert u_j \Vert_{ L^p(B_R(j))}\leq C_u$.
 There exists a thick set $S \subset \RR^d$ such that
 \begin{equation}\label{eq:positivity}
  \sum_{j \in \ZZ^d} u_j \geq \mathbf{1}_S,
 \end{equation}
 $\mathbf{1}_S$ being the characteristic function of $S$.
 \end{enumerate}

We define the random potential $V_\omega$, and the alloy type Hamiltonian $H_\omega$ on $L^2(\RR^d)$ as
 \begin{equation}\label{eq:potential+operator}
V_\omega =\sum_{j \in \ZZ^d} \pi_j(\omega) u_j,
\quad
\text{and}
\quad
H_\omega = - \Delta + V_\omega,
\quad
\omega \in \Omega.
 \end{equation}
The potential $V_\omega$ describes the forces to which an electron in the interior (bulk) of
a solid is exposed due to the atoms placed at lattice points $j\in\ZZ^d$ with modulated interaction parameter $\pi_j$.
The assumptions \hypU and \hypPi imply self-adjointness of $H_\omega$ for \emph{all} $\omega \in \Omega$.
Indeed, there exists a $C_V<\infty$ (independent of $\omega$) such that:
\begin{align*}
 \langle \varphi, V_\omega \varphi\rangle
 \leq
 \lvert \mup \rvert \langle \varphi,\sum_{j \in \ZZ^d}  u_j\varphi\rangle
 &\leq
 \frac{1}{2}\langle \varphi, -\Delta \varphi\rangle  +C_V \Vert \varphi\Vert^2, \\
\Vert V_\omega \varphi\Vert
 &\leq
 \frac{1}{2}\Vert-\Delta \varphi \Vert  +C_V \Vert \varphi\Vert,
\end{align*}
see for instance Theorem XII.96 in \cite{ReedS-78} and its proof.
Consequently by the Kato-Rellich theorem the operator $H_\omega = - \Delta + V_\omega$
is self-adjoint on the domain of $\Delta$ and essentially self-adjoint on $C_c^\infty(\RR^d)$.

Also note that one can without loss of generality put any positive constant $c>0$ in front of $\mathbf{1}_{S}$ in~\eqref{eq:positivity} since this simply amounts to rescaling the distributions $\mu_j$.
In the special case where all $u_j$ are translates of a single function, that is  $u_j(x)=u(x-j)$,
it suffices to find \emph{one} positive measure set $T \subset \Lambda_1(0)$ with $u\geq \mathbf{1}_{T}$
in order to satisfy \eqref{eq:positivity}.
 This $T$ could for instance be a Smith--Volterra--Cantor set, illustrating the novelty of the results in this article,
even in the ergodic setting.
Conversely, if $0 \leq u \in L^p$ is not almost everywhere zero, there exist an $\epsilon>0$ and a positive measure set $T$, such that $u \geq \epsilon \chi_T$,
implying (2) for $u_j(x)=u(x-j)$, up to the universal prefactor $\epsilon$.

In order to state our first result, we define $s \colon [0,\infty) \to[0,1]$
\begin{equation}
\label{definition-s-mu-epsilon}
s(\varepsilon)
=
\sup_{j \in \ZZ^d}
\sup_{E \in \RR}
\mu_j \left( \left[ E - \varepsilon/2, E + \varepsilon/2 \right] \right),
\end{equation}
the uniform modulus of continuity of the family of marginal probability measures $\{\mu_j\}_{j\in \NN}$.
For $L>0$ and $x \in \RR^d$ we denote by $\Lambda_L(x) = x + (- L/2, L/2)^d$ the cube with side length $L$, centered at $x$, and simply write $\Lambda_L = \Lambda_{L}(0)$.
We write $H_{\omega, L, x}$ for the self-adjoint restriction of the Schr\"odinger operator $H_\omega$ to the cube $\Lambda_L(x)$ with periodic, Dirichlet or Neumann boundary conditions.
We also use the notation $\chi_I(A)$ for the spectral projector of a self-adjoint operator $A$ onto the set $I \subset \RR$.

\begin{thm}\label{thm:optWE}
Let $H_\omega$ satisfy \hypU and \hypPi above.
Then, for every $E_0 \in \RR$ there exists $C_W:=C_W(E_0) > 0$, such that,
for all $x \in \RR^d$, all $L \geq \max\{a_1, \ldots, a_d \}$, where $a = (a_1, \dots, a_d) \in (0, \infty)^d$ is from the definition of thickness of $S$,
and
for all intervals $[E - \varepsilon, E + \varepsilon ] \subset (- \infty, E_0]$,
the following Wegner estimate holds
\begin{equation}
\label{eq:WE2}
\EE\{\Tr [ \chi_{[E-\varepsilon,E+\varepsilon]}(H_{\omega, L,x}) ]\}
\le
C_W
s(\varepsilon)
L^d.
\end{equation}
\end{thm}

We next establish that this statement is sharp by showing that a Wegner estimate cannot hold if condition \hypPi is omitted.
Indeed, for the Wegner estimate, it is essential that a level set (to some positive value) of the overall potential contains  a thick set.
The contraposition is the following condition \hypNoPi which is a \emph{non-thick upper bound} on (any level set of) the potential.\\

\begin{enumerate}
 \item[\hypNoPi]
 The sequence $(u_j)_{j \in \NN}\subset L^p(\RR^d)$ of functions $u_j\colon  \RR^d \to [0, \infty)$ entering \eqref{eq:potential+operator} is such that
 \[
 U := \sum_{j \in \ZZ^d} u_j
 \]
 satisfies $\lVert U \rVert_{L^\infty(\RR^d)} \leq C_U$ for some $C_U > 0$, and for no $\kappa > 0$, the set $S_\kappa := \{ x \in \RR^d \colon U(x) \geq \kappa \}$ is thick.
\end{enumerate}

Note that in the ergodic setting $u_j(x)=u(x-j)$, this conditions implies that the function $u$ is zero almost everywhere.
The following theorem shows that under Assumption \hypNoPi  no translation uniform Wegner estimate can hold.
The reason is that in this case there is an increasing sequence of cubes with "stubborn" eigenvalues which are arbitrarily insensitive to the potential:

\begin{thm}
	\label{thm:insensitive}
	Let $H_\omega$ satisfy \hypU and \hypNoPi.
	Denote by $H_{\omega, L, x_j}$ the restriction of $H_\omega$ onto $\Lambda_L(x_j)$ with Dirichlet boundary conditions.
	Then, for every $E \geq 0$ and for all $L \geq 1$ there are infinitely many mutually disjoint cubes $\Lambda_L(x_j)$, $j \in \NN$
    such that for all configurations $\omega$
	\begin{equation}\label{eq:stable-eigenvalue}
		\sigma( H_{\omega, L, x_j})
		\cap [E - \varepsilon, E + \varepsilon]
		\neq
		\emptyset
	\quad
    \text{ with } \varepsilon=12\pi \sqrt{E+1}/L.
	\end{equation}
\end{thm}

Note that the width $2\varepsilon$ of the energy interval in \eqref{eq:stable-eigenvalue} is proportional to $1/L$.
But in order to prove Anderson localisation using multiscale analysis it would suffice to consider interval lengths  $\varepsilon\sim \exp(-L^\beta)$ for some $\beta \in(0,1)$ in a Wegner estimate.
Hence, at a first glance, Theorem \ref{thm:insensitive} might leave the possibility that with differently chosen interval lengths this phenomenon of stubborn eigenvalues might disappear.
However, we can also rule this out.
Note that Theorem~\ref{thm:insensitive} is about \emph{all} energies $E \geq 0$.
If we focus on neighborhoods of eigenvalues of $H_{0, L, 0}$ we can strengthen the statement and find stubborn eigenvalues in smaller intervals.
The following theorem makes this precise and shows that without Assumption \hypPi there cannot be any Wegner estimate with interval lengths that are exponentially small in $L$ and thus no Wegner estimate useful in the multiscale analysis.

\begin{thm}
	\label{thm:insensitive_2}
	Let $H_\omega$ satisfy \hypU and \hypNoPi.
	Denote by $H_{\omega, L, x_j}$ the restriction of $H_\omega$ onto $\Lambda_L(x_j)$ with Dirichlet boundary conditions.
	Then, for all $L \geq 1$ and for any eigenvalue $E$ of $H_{0, L, 0}$
    there are infinitely many mutually disjoint cubes $\Lambda_L(x_j)$, $j \in \NN$
    such that for all configurations $\omega$
	\begin{equation}\label{eq:stable-eigenvalue_exponential}
		\sigma( H_{\omega, L, x_j})
		\cap [E - e^{-L}, E + e^{-L}]
		\neq
		\emptyset.
	\quad
	\end{equation}
\end{thm}

The proofs of Theorem~\ref{thm:insensitive} and~\ref{thm:insensitive_2} can be found in Section~\ref{subsec:insensitive}.

\begin{rem}
 Non-ergodic random Schr\"odinger operators have been investigated by several authors in the recent years~\cite{RojasMolina-12, RojasMolinaV-13, Klein-13, TaeuferT-17, NakicTTV-18, SeelmannT-20}.
 One underlying objective in these works has been the identification of minimal assumptions on random Schr\"odinger operators which still ensure localisation.
 This has not only served the purpose of expanding the class of models but it has also helped to better distinguish between necessary assumptions and technical assumptions arizing from the particular method of proof.

 While in the \emph{ergodic} setting, it is intuitive that a minimal assumption should be that the single-site potential $u$ is not identically zero, our contribution in this paper can be understood to firstly generalize this to Assumption \hypPi, which is also valid in the non-ergodic setting,
 and, secondly, to prove that this indeed still leads to Wegner estimates and localisation.
\end{rem}

\begin{rem}[Sparse potentials]
Theorems~\ref{thm:optWE} to~\ref{thm:insensitive_2}
concern the Hamiltonian \eqref{eq:potential+operator} governing the motion of an electron in the bulk of a solid.
In the literature on random operators also surface interactions of the form
\begin{equation}\label{eq:surface}
V_\omega =\sum_{j \in \ZZ^D} \pi_j(\omega) u_j,
\quad
\text{with }
\quad
0<D<d
 \end{equation}
exhibiting an (anisotropic) spacial decay (on average), are studied.
In fact, they are a special case of so called
\emph{sparse (random) potentials}, which can be modeled in various ways, see for instance
\cite{Krishna-93,HundertmarkK-00,KostrykinS-01b,Kirsch-02,KirschV-02c,Krishna-02,FigotinGKM-07}.
One way is to deterministically dilute the lattice sites to which the single site potentials are attached
leading to a stochastic field of the type
\[
V^{\mathrm{sprs}}_\omega(x)= \sum_{k\in \Gamma} \pi_k(\omega) u(x-k)
\quad \text{ with }
\lim_{L\to\infty} \frac{\sharp \Gamma \cap \Lambda_L }{|\Lambda_L |} =0.
\]
The random surface potential spelled out in \eqref{eq:surface} falls into this class.
Another way would be to dilute the sites probabilistically
\[
V^{\mathrm{sprs}}_\omega(x)= \sum_{k\in \ZZ^d} \pi_k(\omega) u(x-k)
\quad \text{ with }
\lim_{|k|\to\infty} \PP\left\{ \pi_k=0 \right\} =1.
\]
Finally one can consider an arrangement of interactions $u_k$ on a lattice which become weaker with increasing distance from the origin
\begin{multline*}
  V^{\mathrm{sprs}}_\omega(x)= \sum_{k\in \ZZ^d} \pi_k(\omega) u_k(x)
\quad \text{ with } \supp u_k \subset \Lambda_R(k) \text{ for some $R$}
 \\
\text{ and }
\lim_{\lvert k \rvert \to\infty} u_k(\cdot +k) =0
\text{ in some appropriate sense.}
\end{multline*}
Let us concentrate on the last model and ask how to distinguish between sparse potentials on one hand and macroscopic/bulk ones on the other.
Our two complementary Theorems \ref{thm:optWE} and \ref{thm:insensitive} show that (provided \hypU holds) condition \hypPi
distinguishes between the two cases: If \hypPi is satisfied one obtains translation uniform Wegner estimates,
if not there is a sequence of larger and larger cubes escaping to infinity with eigenvalues of the local Hamiltonian insensitive to randomness.
The latter means that the disorder present in the potential is too weak to efficiently influence eigenvalues in certain spatial sectors,
which is a signature of sparse potentials.
\end{rem}

Let us now turn back to Hamiltonians as in \eqref{eq:potential+operator} satisfying \hypU and \hypPi.
If the family $\{ H_\omega \}_{\omega \in \Omega}$ is even ergodic, cf.~\cite{CarmonaL-90,PasturF-92,Veselic-08},
the thermodynamic limit
\[
 N \colon \RR \to [0 ,\infty),
 \quad
 N(\cdot)
 =
 \lim_{L \to \infty}
 \frac{\Tr [ \chi_{(- \infty,\cdot]}(H_{\omega, L,x}) ]}{L^d}
\]
of the normalized eigenvalue counting functions exists for all almost all $\omega \in \Omega$ and $x \in \RR^d$, is a distribution function, and is called the integrated density of states (IDS).
If the IDS exists and $\lim_{\varepsilon\to0}s(\varepsilon)=0$,  Theorem~\ref{thm:optWE}
implies continuity of the IDS:

\begin{cor}
 Let $H_\omega$ satisfy \hypU and \hypPi above.
 Assume furthermore that the random family $\{ H_\omega \}_{\omega \in \Omega}$ is ergodic.
 Then, the integrated density of states $N$ exists almost surely. For every $E_0\in  \RR$ and $\varepsilon \in (0,1]$ one has
 \[
  0 \leq N(E_0) - N(E_0-\varepsilon) \leq C_W(E_0) s(\varepsilon).
 \]
\end{cor}
If the random variables $\pi_j $ have a H\"older continuous density, more precisely if for some $\alpha > 0$
\begin{equation}\label{eq:Hoelder}
\sup_{\varepsilon>0} s(\varepsilon)/\varepsilon^\alpha=:C_\alpha<\infty
\end{equation}
then the Wegner estimate of Theorem~\ref{thm:optWE} can also be used to prove Anderson and dynamical localisation at the bottom of the spectrum via the multiscale analysis.
Actually, H\"older continuity could be relaxed to sufficiently strong log-H\"older continuity, cf.~e.g.~Lemma 4.6.2. in~\cite{Veselic-08}.
There exists an entire hierarchy of notions of localisation. We spell out only three of them and refer for more details to~\cite{GerminetK-01}.
\begin{dfn}
 The random family of operators $\{ H_\omega \}_{\omega \in \Omega}$ exhibits
 \begin{itemize}
 \item
 \emph{Anderson localisation} in $I \subset \RR$ if $\PP$-almost surely the spectrum of $H_\omega$ within $I$ is only of pure point type with exponentially decaying eigenfunctions,
 \item
 \emph{dynamical localisation} in $I$ if for $\PP$-almost every $\omega \in \Omega$, every compact interval $J \subset I$, and every $\psi \in L^2(\RR^d)$ with compact support we have
 \[
  \sup_{t \in \RR}
  \lVert
  (1 + \lvert x \rvert)^{n/2}
  \chi_J(H_\omega) \mathrm{e}^{- i t H_\omega} \psi
  \rVert^2
  <
  \infty
  \quad
  \text{for all $n \geq 0$},
 \]
 \item
 and \emph{strong Hilbert--Schmidt dynamical localisation} in $I$ if for every compact interval $J \subset I$,
 we have
 \begin{equation*}
 \sup_{y\in\RR^d} \EE\left( \sup_{f\in L^\infty(\RR), \|f\|_\infty\leq 1}
  \lVert (1 + \lvert x-y \rvert)^{n/2}   f(H_\omega)  \chi_J(H_\omega) \mathbf{1}_{\Lambda_1(y)}  \rVert_{HS}^2\right)
  <
  \infty
 \end{equation*}
 for all $n \geq 0$.
 Here $(1 + \lvert x-y \rvert)$ is understood as the multiplication operator with the function $x \mapsto (1 + \lvert x-y \rvert)$.
 \end{itemize}
\end{dfn}

Strong Hilbert-Schmidt dynamical localisation implies dynamical localisation which implies Anderson localisation.
In general, the reverse is not true, but for a natural class of random Hamiltonians these three properties turn out to be actually equivalent, see~\cite{GerminetK-04} for details.

If the (lower-bounded) random operator $\{ H_\omega \}_{\omega \in \Omega}$ is ergodic, it exhibits almost sure spectrum, and thus a well defined spectral minimum,
which is in fact a fluctuation boundary in the sense of Lifschitz tails. Since we do not assume that the random variables $\{ \pi_j\}_{j\in \ZZ^d},$ are identically distributed, we are in a more general
situation. To ensure that there is no spectrum below zero, that $0\in \sigma(H_\omega)$ with positive probability, and it is a (generalized) fluctuation boundary we need a further\\

\begin{itemize}
\item[\hypO]
For all $j \in \ZZ^d$ we have $\min \supp \mu_j =0 $.
\end{itemize}

This implies that zero is the overall minimum of the spectrum in the following sense:
\begin{equation}
 \label{eq_minimum-sigma}
 \sup\{E \in \RR \mid E \leq \min \sigma(H_\omega)\ \text{$\PP$-almost surely} \}=0,
\end{equation}
see Section~\ref{subsec:spectral_minimum} for a proof.

\begin{rem}
If we require additionally to Assumption \hypO that there exist a distribution $\mu\colon \cB(\RR)\to [0,1]$ and an $\varepsilon_0>0$ such that
\[
\min \supp \mu =0 \quad \text{ and } \quad
\forall \varepsilon \in(0,\varepsilon_0) , j \in \ZZ^d\colon \quad
\mu_j([0,\varepsilon]) \geq \mu([0,\varepsilon])
\]
then $\min \sigma(H_\omega)= 0$ almost surely by a simple Borel-Cantelli argument.
\end{rem}

\bigskip

 \begin{thm}
 \label{thm:loc}
 Let Assumptions \hypU, \hypPi, and \hypO, as well as H\"older continuity \eqref{eq:Hoelder} hold.
 Then there exists an $E_+ >0$ such that the operator $H_\omega$ exhibits strong Hilbert-Schmidt dynamical localisation in $[0,E_+]$.
\end{thm}

 Theorem~\ref{thm:loc} is proved by the multiscale analysis, see for instance~\cite{Stollmann-01,GerminetK-01, GerminetK-03}. More specifically, in our situation, Theorem 2.3 in \cite{RojasMolina-12} applies directly and yields
 strong Hilbert--Schmidt dynamical localisation.
Apart from the Wegner estimate, Theorem~\ref{thm:optWE}, it requires the following (non-stationary) initial scale estimate:

\begin{thm}
 \label{thm:ISE}
 Let Assumptions \hypU,  \hypPi, and \hypO, as well as $\lim_{\varepsilon\to0}s(\varepsilon)=0$ hold.
 \\
Then there exist $c_0,L_0 > 0$ such that for all $L \geq L_0$, all $x \in \RR^d$, and all cubes $A,B\subset \Lambda_L(x)$  with $\dist(A,B)\ge L/3$ we have
  \begin{equation*}
    \PP
     \left\{
	      	\lVert \mathbf{1}_A (H_{\omega, L, x}-L^{-1/2})^{-1} \mathbf{1}_B \rVert
    	   	\le
        	\exp \left( -c_0 L^{1/2} \right)
      \right\}
      \ge
      1
      -
      \exp \left( - c_0 L^{d/4} \right).
  \end{equation*}
 \end{thm}

\section{Proofs}
\subsection{Spectral inequality and Wegner estimate}

A core idea for the Wegner estimate is that modifying random variables by a fixed positive number $\delta > 0$ will make eigenvalues go up by an amount proportional to $\delta$,
\emph{independently of the side length $L$ of the box}.
This will then imply that the probability of finding an eigenvalue in an interval is (at most) proportional to the uniform modulus of continuity of said interval -- i.e. the Wegner estimate.
In other words, one has to prove that the operator of multiplication by the potential difference $V_{\omega + \delta} - V_\omega$, where $V_{\omega + \delta}$ denotes the configuration with all random variables increased by $\delta$, is not only a non-negative operator, but that its effect on the part of the spectrum under consideration is the one of a strictly positive operator.
This is the statement of quantitative unique continuation principles for spectral subspaces of Schr\"odinger operators.
Typically, unique continuation principles for spectral subspaces of Schr\"odinger operators require that the multiplication operator is lower bounded by a characteristic function of an open set~\cite{CombesHK-03, RojasMolinaV-13, Klein-13, NakicTTV-18}.
One notable exception {where a stronger statement is known, and which we build upon in this note, is the situation of spectral subspaces of the pure Laplacian}.
In this case it suffices to have a multiplication by the characteristic function of a thick set.
The following theorem makes this precise.
It is implied by \cite{EgidiV-20} and has been spelled out in Section 5 of \cite{EgidiV-18}.

\begin{thm}\label{thm:uncertainty}
Assume that $S$ is a $(\gamma,a)$-thick set and that $L>0$ satisfies $A_a\subset [0,L]^d$. Let $x\in \RR^d$.
Let  $H_{0, L, x}$ denote the negative of the Laplace operator
on $\Lambda_L(x)$ with periodic, Dirichlet or Neumann boundary conditions.
Then there is an absolute constant $K\geq 1$ such that for all $E \geq 0$, and all $f \in L^2(\Lambda_L)$ we have
\begin{equation*}
 \|\chi_{(-\infty, E]}(H_{0, L, x}) f\|_{L^2(\Lambda_L)}^2
\leq
\left(\frac{K^d}{\gamma}\right)^{ K\sqrt{E}(|a|_{1}+d)}
 \mkern-27mu
\|\chi_{(-\infty, E]}(H_{0, L, x}) f\|_{L^2(\Lambda_L\cap S)}^2.
 \end{equation*}
\end{thm}

Theorem~\ref{thm:uncertainty} is a finite volume variant of the Logvinenko-Sereda theorem~\cite{LogvinenkoS-74,Kacnelson-73} in a version of Kovrijkine~\cite{Kovrijkine-00, Kovrijkine-01}.
It can be interpreted as a quadratic form inequality
\begin{equation}
 \label{eq:quadratic_form_Log-Sereda}
 \chi_{(-\infty, E]}(H_{0, L, x})
 \leq
 \left(\frac{K^d}{\gamma}\right)^{ K\sqrt{E}(|a|_{1}+d)}
 \mkern-18mu
 \chi_{(-\infty, E]}(H_{0, L, x})
 \mathbf{1}_S
 \chi_{(-\infty, E]}(H_{0, L, x}).
\end{equation}
Clearly, on can replace all spectral projectors in~\eqref{eq:quadratic_form_Log-Sereda} by $\chi_{J}(H_{0, L, x})$ for any $J \subset (- \infty, E]$, since one can always multiply the inequality with these projectors from both sides.

\begin{proof}[Proof of Theorem \ref{thm:optWE}]
The statement is proved using the strategy of \cite{CombesHK-07}: We split the trace in~\eqref{eq:WE2} in two contributions according to spectral projectors of the unperturbed operator $H_{0, L, x}$, that is the pure negative Laplacian on $L^2(\Lambda_L(x))$.
For that purpose let $I:=[E-\varepsilon,E+\varepsilon]$ and let
$J \supset I$ be another interval with $\dist(I, J^c)>0$.
Then
\begin{equation}
\label{eq:decomp}
\Tr [ \chi_{I}(H_{\omega, L, x}) ]
=
\Tr [ \chi_{I}(H_{\omega, L, x}) \chi_{J}(H_{0, L, x}) ]+ \Tr [ \chi_{I}(H_{\omega, L, x}) \chi_{J^c}(H_{0, L, x}) ].
\end{equation}
The strategy of~\cite{CombesHK-07} relies on estimating the expectation of the two terms on the right hand side of~\eqref{eq:decomp} separately:
 The second term on the right hand side of~\eqref{eq:decomp} can be estimated as in (2.6) to (2.20) of~\cite{CombesHK-07} by using the fact that $\supp {u_j} \subset B_R(j)$ and Combes-Thomas bounds.
 This leads to
 \[
  \EE
    \left\{
      \Tr [ \chi_{I}(H_{\omega, L, x}) \chi_{J^c}(H_{0, L, x}) ]
    \right\}
  \leq
  C_1
  \EE\{\Tr [ \chi_{I}(H_{\omega, L,x}) ]\}
  +
  C_2
  s(\varepsilon) L^d,
 \]
 where $C_1,C_2 > 0$ are constants and $C_1 < 1$, hence the corresponding term can be absorbed on the left hand side of \eqref{eq:decomp}.

 The expectation of the first term on the right hand side of~\eqref{eq:decomp} is estimated in (2.21) to (2.32) of~\cite{CombesHK-07} by using a scale-free quantitative unique continuation principle for spectral projectors.
 In their situation this relies on the fact that their sum $\sum_j u_j$ is periodic and uniformly positive on an open set of positive measure, cf.~\cite[Theorem 2.1]{CombesHK-07}.
 In our case, these assumptions of periodicity and positivity on an open set no longer hold, but  Theorem~\ref{thm:uncertainty} yields an appropriate replacement and leads to the quadratic form inequality
 \begin{align*}\label{eq:quadratic-forms}
\chi_{J}(H_{0, L, x})
&\leq
C(J, S)
\cdot
\chi_{J}(H_{0, L, x})
\mathbf{1}_S
\chi_{J}(H_{0, L, x})
\\
&\leq
C(J,S)
\sum_{j \in \ZZ^d}
\chi_{J}(H_{0, L, x})
u_j
\chi_{J}(H_{0, L, x})
\end{align*}
where it is implicitly understood that $\sum_{j \in \ZZ^d} u_j$ is restricted to $\Lambda_L(x)$.
 This replaces the inequality of \cite[Theorem 2.1]{CombesHK-07}.
The rest of the proof follows along the lines of \cite[Section 2]{CombesHK-07}.
\end{proof}

\subsection{Spectral minimum}
\label{subsec:spectral_minimum}

Let us prove Identity~\eqref{eq_minimum-sigma}.
Assume $\min \supp \mu_j =0 $ for all $j \in \ZZ^d$ and let $\varepsilon\in(0,1)$. There exist $r< \infty$ as well as an $L^2$-normalized $\varphi \in C_c^\infty(\RR^d)$
with support contained in $B_r(0)$ and such that
\[
 \langle \varphi, -\Delta \varphi\rangle =\Vert\nabla \varphi\Vert^2 < \varepsilon/3.
\]
Set
\[
\Omega_\varepsilon:= \pi_Q^{-1} \left( \bigtimes_{j \in Q} \left[ 0,\frac{\varepsilon}{3 C_V} \right] \right)
\]
where $Q =\ZZ^d \cap B_{2r +2R}(0)$ and $\pi_Q \colon \Omega \to \bigtimes_{j \in Q} [\ml, \mup]$ is the canonical projection.
Then by assumption $\PP(\Omega_\varepsilon) > 0$ and we have for all $\omega \in \Omega_\varepsilon$
\begin{equation*}
 \langle \varphi, V_\omega \varphi\rangle
 \leq  \frac{\varepsilon}{3 C_V} \left (\Vert\nabla \varphi\Vert^2/2 + C_V\Vert\varphi\Vert^2 \right)
 \leq
 \frac{\varepsilon}{3 C_V} \left ( \frac{\varepsilon}{6}  + C_V \right)
 <
 \frac{\varepsilon}{2}
 \end{equation*}
 where we assumed without loss of generality $C_V \geq 1$.
Thus $\langle \varphi, H_\omega \varphi\rangle < \varepsilon$ and by the variational characterization $\min\sigma(H_\omega)< \varepsilon $ for all $\omega \in\Omega_\varepsilon$.
Since this holds for all $\varepsilon \in (0,1)$, we obtain \eqref{eq_minimum-sigma}.

\subsection{Initial length scale estimate}
We provide the proof of the initial length scale estimate for our non-stationary model, Theorem~\ref{thm:ISE}.
For this purpose, we show how our random potential can be lower bounded by a simpler one to which the results of \cite{SchumacherV}
apply directly.

First we modify the random variables.
Since  $\min \supp \mu_j =0 $ for all $j \in \ZZ^d$ and $\lim_{\varepsilon\to0}s(\varepsilon)=0$
there exist $\varepsilon_1>0$ such that $0< s(\varepsilon_1)<1$.
For all $j \in \ZZ^d$ define $\eta_j \colon \Omega \to \RR$ by
\[
 \eta_j = \varepsilon_1 \mathbf{1}_{\{\pi_j\in [\varepsilon_1, \mup]\}},
\]
in particular $0\leq \eta_j \leq \pi_j$.
Let $A_a$ be the window in the definition of the thick set $S$.
Choose $L\in \NN$ with $L\geq \max\{a_1, \ldots, a_d\}$.
Then
\[
|S\cap \Lambda_L(k)| \geq \gamma a_1 \cdots a_d  =\frac{\gamma a_1 \cdots a_d}{L^d} L^d=: \tg L^d ,
\]
hence $S$ is $(\tg, (L,\ldots, L))$-thick.
For $k\in \ZZ^d$ set $\Lambda^+(k)=\Lambda_{L+2R}(k)\cap \ZZ^d$, $N:=\sharp \Lambda^+(k)$, and for $j \in \Lambda^+(k)$ set
\[
S_j:=\{y\in \Lambda_L(k)\mid u_j(y) \geq 1/N\}.
\]
Note that by translation invariance $N$ is $k$-independent.
Then
\[
\bigcup_{j\in \Lambda^+(k)}S_j \supset S \cap \Lambda_L(k).
\]
Subadditivity gives
\[
\tg L^d \leq \lvert \bigcup_{j\in \Lambda^+(k)}S_j \rvert \leq
\sum_{j\in \Lambda^+(k)} |S_j|,
\]
hence there exist a $j_k\in \Lambda^+(k)$ such that $|S_{j_k}| \geq \tg L^d/ N$.
It is possible to find a subset $T_k\subset S_{j_k}\subset\Lambda_L(k)$ such that $|T_{k}| =\hg :=\tg L^d/ N$.

Thus we have identified  for each $k \in \RR^d$ a single site potential $u_{j_k}$ with the following properties:
\begin{align}
u_{j_k}\geq \frac{1}{N} \mathbf{1}_{T_{k}},  \quad  |T_{k}|=\hg .
\end{align}

We choose the sublattice $\Gamma=\big((L+2R)\ZZ\big)^d$
with periodicity cell $\Lambda_{L+2R}$. The  lattice $\Gamma$ is sufficiently sparse, so that the map $\Gamma\ni k\mapsto j_k$ is injective, hence the random variables $\pi_{j_k} , k\in \Gamma$ are independent.
Now compare the original operator with a diluted Anderson model
\[
V_\omega \geq \sum_{k\in \Gamma} \pi_{j_k}(\omega) u_{j_k}
\geq
\sum_{k\in \Gamma} \eta_{j_k}(\omega) \frac{1}{N} \mathbf{1}_{T_k}=: W_\omega .
\]

The initial scale estimate, Theorem~\ref{thm:ISE} will follow from an application of Theorem~3.3 and its Corollary~4.2 in~\cite{SchumacherV}.
For this purpose we need to verify the non-degeneracy condition in Definition 3.1 of~\cite{SchumacherV}.

Since  $W_\omega \leq V_\omega$, it is sufficient to establish this condition for $W_\omega$.
For this purpose let us  define
\begin{alignat*}{2}
\eta_{j_k}&\colon \Omega_{W}:= \RR^{\Gamma} \to \RR,
\quad
&&\eta_{j_k}(\omega)=\eta_{j_k}(\pi(\omega)_{j_k}),
\\
\lambda_k &\colon \Omega_{W}\to \Omega_0 =\Gamma\times \Omega_{W},
\quad
&&\lambda_k(\omega)= (k,\omega),
\\
u &\colon\Omega_0 \times \RR^d \to \RR,
\quad
&&u\big((k,\omega),x) =\frac{1}{N}  \eta_{j_k}(\omega ) \mathbf{1}_{T_k}(x+k) .
 \end{alignat*}
If we choose  $m :=\min\{\varepsilon_1/N,\hg, 1-s(\varepsilon_1) \}>0$ it follows directly that $W_\omega$ is
${m}$-non-degenerate, more precisely
      \begin{equation*}
        \inf_{k\in\ZZ^d}\PP \bigl\{
          \vol{\{x\in\Lambda_{L+2R} \mid u_{\lambda_k}(x)\ge{m}\}}\ge{m} \bigr\} \ge{m}
      \end{equation*}
which is the condition formulated in Definition 3.1 of~\cite{SchumacherV}.
Now Corollary 4.2 of \cite{SchumacherV} immediately yields Theorem \ref{thm:ISE}.

\subsection{Proof of Theorems~\ref{thm:insensitive} and~\ref{thm:insensitive_2}}

	\label{subsec:insensitive}

In this section we prove Theorems~\ref{thm:insensitive} and~\ref{thm:insensitive_2}, stating that \hypNoPi, the contraposition of Assumption \hypPi, excludes a Wegner estimate.

\begin{proof}[Proof of Theorem~\ref{thm:insensitive}]
	First, observe that \hypNoPi implies that for every choice $L, \delta, \kappa > 0$ there are infinitely many mutually disjoint cubes $\Lambda_L(x_j) \subset \RR^d$, $j \in \NN$, such that
	\[
	\vol ( S_\kappa \cap \Lambda_L(x_j)) < \delta
	\quad
	\text{for all $j \in \NN$}.
	\]
	Indeed, if there were only finitely many such cubes, then they would all be contained in a larger cube $\Lambda_{\tilde L}(0)$.
	But then, for every $x \in \RR^d$, the cube $\Lambda_{\tilde L + 2 L}(x)$ would have to contain at least one cube $\Lambda_L(y) \subset \Lambda_{\tilde L + 2 L}(x)$ with $\vol (S_\kappa \cap \Lambda_L(y)) \geq \delta$.
	Hence, $\vol (S_\kappa \cap \Lambda_{\tilde L + 2 L}(x) ) \geq \delta$ for all $x \in \RR^d$ and $S_\kappa$ would be thick, a contradiction.
	
	Now, let $E \geq 0$. The maximal gap between two successive eigenvalues of $H_{0,L,0}$ below $E + 1$ tends to zero as $L \to \infty$.
    More precisely, the Dirichlet eigenvalues on the cube $\Lambda_L$ have the representation
 \begin{equation*}
    E_n(L)=\frac{\pi^2}{L^2}\sum_{j=1}^{d} n_j^2, \quad n \in \NN^d
 \end{equation*}
from which it is easy to see that
 \begin{equation*}
   \operatorname{dist} (E, \sigma(  H_{0, L, x_j})) =\operatorname{dist} (E, \sigma(  H_{0, L, 0})) \leq 6\pi \sqrt{E+1}/L=:\frac{\varepsilon}{2}
 \end{equation*}
	Take a normalized eigenfunction $\varphi$ of $H_{0,L,x_j}$ to an eigenvalue in an $\frac{\varepsilon}{2}$-neighborhood of $E$.
	Then
	\[
	\lVert
		\left(
			H_{0,L,x_j} - E
		\right)
		\varphi
	\rVert_{L^2(\Lambda_L(x_j))}
	\leq
	\frac{\varepsilon}{2}.
	\]
	Since $\lvert V_\omega \rvert \leq \max \{ \lvert \ml \rvert, \lvert \mup \rvert \} U(x)$, we have for almost every $\omega \in \Omega$
	\[
	\lVert V_\omega \varphi \rVert_{L^2(\Lambda_L(x_j))}
	\leq
	\max \{ \lvert \ml \rvert, \lvert \mup \rvert \}
	\lVert
	U
	\varphi
	\rVert_{L^2(\Lambda_L(x_j))},
	\]
	and estimate
	\begin{align*}
	\lVert
	U
	\varphi
	\rVert_{L^2(\Lambda_L(x_j))}
	&\leq
	\lVert U \rVert_{L^4(\Lambda_L(x_j))}
	\lVert \varphi \rVert_{L^4(\Lambda_L(x_j))}
	\\
	&=
	\lVert U \rVert_{L^4(\Lambda_L(x_j))}
	\lVert \euler^{- H_{0, L, x_j}} \euler^{H_{0, L, x_j}} \varphi \rVert_{L^4(\Lambda_L(x_j))}.
	\end{align*}
The exponential  $\euler^{- t H_{0,L,x_j}}$ of the Dirichlet Laplacian (heat semigroup at time one)
is a convolution operator with corresponding integral kernel bounded by the free kernel
	\[
	p_1(x-y)
	=
	\frac{1}{(4 \pi)^{d/2}}\exp
	\left(
	\frac{\lvert x - y \rvert}{4}
	\right)
	\]
	see, for instance the proof of Lemma 2.2 in~\cite{TaeuferV-15}.
	In particular, the function $\RR^d \ni z \mapsto p_1(z)$ is in every $L^p$-space for $p \in [1,\infty)$.
	By Young's convolution inequality, we can therefore estimate for every $g \in L^2(\Lambda_L(x_j))$
	\begin{align*}
	\lVert \euler^{- H_{0, L, x_j}} g \rVert_{L^4(\Lambda_L(x_j))}
	\leq
	\lVert
	p_1 \ast j_{\Lambda_L(x_j)} g
	\rVert_{L^4(\RR^d)}
	&\leq
	\lVert
	p_1
	\rVert_{L^{4/3}(\RR^d)}
	\lVert
	g
	\rVert_{L^2(\Lambda_L(x_j))}
	\\
	&=
	C_d
	\lVert
	g
	\rVert_{L^2(\Lambda_L(x_j)}
	\end{align*}
	for an $L$-independent $C_d > 0$, where $j_{\Lambda_L(x_j)} \colon L^2(\Lambda_L(x_j)) \to L^2(\RR^d)$ is the canonical embedding.
	Using $\lVert \euler^{H_{0, L, x_j}} \varphi \rVert_{L^2(\Lambda_L(x_j))} \leq \euler^{E + \frac{\varepsilon}{2}}$, we conclude
	\[
	\lVert V_\omega \varphi \rVert_{L^2(\Lambda_L(x_j))}
	\leq
	C_d
	\max \{ \lvert \ml \rvert, \lvert \mup \rvert \}
	\euler^{E + \frac{\varepsilon}{2}}
	\lVert U \rVert_{L^4(\Lambda_L(x_j))}
    \leq  C \lVert U \rVert_{L^4(\Lambda_L(x_j))}
	\]
	for some  $C=C(d,m_+,m_-,E)$ if $L\geq 1$.
    We split $U = U \mathbf{1}_{S_\kappa} + U \mathbf{1}_{S_\kappa^c}$, and find
	\[
	\lVert U \rVert_{L^4(\Lambda_L(x_j)}
	\leq
	C_U
	\lVert
		\mathbf{1}_{S_\kappa}
	\rVert_{L^4(\Lambda_L(x_j))}
	+
	\kappa
	L^{d/4}
	\leq
	C_U
	\delta^{1/4}
	+
	\kappa
	L^{d/4}.	
	\]
	Choosing
\begin{equation*}
  \delta\leq \left(\frac{6\pi \sqrt{E+1}}{2C_UCL} \right)^4
  \quad \text{ and } \quad
  \kappa \leq \frac{6\pi \sqrt{E+1}}{2C L^{1+d/2}}
\end{equation*}
this sum can be made smaller than $\frac{\varepsilon}{2C }=3\pi \sqrt{E+1}/(C \, L)$ whence	
	\[
	\lVert \left( H_{\omega, L, x_j} - E \right) \varphi \rVert_{L^2(\Lambda_L(x_j))}
	\leq
	\varepsilon =12\pi \sqrt{E+1}/L.
	\]
	This implies $\lVert \left( H_{\omega, L, x_j} - E \right)^{-1} \varphi \rVert_{L^2(\Lambda_L(x_j))} \geq \varepsilon^{-1}$ which in turn shows
	\[
	\operatorname{dist} (E, \sigma (H_{\omega, L, x_j})) \leq \varepsilon
	\quad
	\text{for  all configurations $\omega \in \Omega$}.
	\qedhere
	\]
\end{proof}

The proof of Theorem~\ref{thm:insensitive_2} proceeds along the lines of the proof of
Theorem \ref{thm:insensitive} with the difference that one chooses
\begin{equation*}
  \delta\leq \left(\frac{6\pi \sqrt{E+1}}{2C_UC e^L} \right)^4
  \quad \text{ and } \quad
  \kappa \leq \frac{6\pi \sqrt{E+1}}{2C L^{d/2}e^L}.
\end{equation*}


\subsection*{Acknowledgment}
Stimulating and helpful discussions with Christoph Schumacher are gratefully acknowledged.

\subsection*{Funding}
Parts of this work were completed while M.T. was working at Queen Mary University of London
and supported by the European Research Council starting grant 639305 (SPECTRUM).
Parts of this work is based on the project
\emph{Multiscale version of the Logvinenko-Sereda Theorem} which was funded under the DFG grant VE 253/7-1.

%


\end{document}